\documentclass[12pt,reqno]{amsart}
\usepackage{a4}
\usepackage{amssymb}
\usepackage{amsmath}
\usepackage{amsthm}
\usepackage{enumerate}

\newcommand{\cc}{c} 
\newcommand{\comporesult}{d} 
\newcommand{\affcu}{\mu} 

\newcommand{\const}{\operatorname{const}}
\newcommand{\lk}[1]{\bar #1} 

\newcommand{\rp}[1]{R_{#1}}
\newcommand{\pp}[1]{P_{#1}^{(0)}}
\newcommand{\qp}[1]{P_{#1}^{(1)}}
\newcommand{\ppp}[2]{P_{#1}^{(#2)}}

\newcommand{\gcvatp}{\gamma} 
\newcommand{\ipt}{p} 
\newcommand{\kidx}{k} 
\newcommand{\atp}[1]{\left. #1 \right|_\ipt}
\newcommand{\fpre}{\varphi}
\newcommand{\gpre}{\psi}
\newcommand{\R}{\mathbb{R}} \newcommand{\Z}{\mathbb{Z}}
 \newcommand{\N}{\mathbb{N}}
\newcommand{\NZ}{\mathbb{N}_0}

\theoremstyle{plain} 
\newtheorem{theorem}{Theorem}
\newtheorem{lemma}[theorem]{Lemma}
\newtheorem{corollary}[theorem]{Corollary}
\theoremstyle{remark}
\newtheorem*{remark*}{Remark}

\begin{document}
\title{An affine characterization of quadratic curves}
\author{Thomas Binder}
\address{University of L{\"u}beck,
Institute of Mathematics, Wallstra{\ss}e 40,
D-23560 L{\"u}beck, Germany}
\email{thmsbinder@gmail.com}

\begin{abstract}
By analyzing the affine Taylor expansion of a non-degenerate
plane curve, we obtain characterizations of classes of such curves
via curvature properties of the gravity curve. The proof is based on
an analysis of the degree parity and leading coefficients
of polynomials occurring in the expansion.
\end{abstract}

\maketitle

By
$\R^2$ we denote the real affine plane.
Let $\cc:I\to\R^2$, where $I\subseteq \R$ is an open interval,
be a differentiable plane affine curve.
We assume that $\cc$ is non-degenerate and parametrized
by affine arclength $s\in I$.

For fixed $p\in I$ the gravity curve $\gcvatp$ of $\cc$ at
$p$ is constructed as follows.
Consider the line $\ell$ obtained by parallel translation
of $\cc'(p)$ by $\delta>0$, in the direction of the halfplane
containing $\cc(I)$. Since $\cc$ is strongly convex at $p$, $\ell$
intersects $\cc(I)$ in exactly two points which define a line
segment in $\ell$. The midpoint of this line segment lies on the gravity
curve $\gcvatp$. When $\delta$ varies these midpoints sweep out all
of the differentiable curve $\gcvatp$,
where $\gcvatp(0)=\cc(p)$ by continuity.
As our considerations are local, we make restrictions to suitable
subsets without mentioning this each time.

Affine differential geometry is motivated by the fact that,
in a generalization to locally strongly convex surfaces and
hypersurfaces, the direction of
$\gcvatp'(0)$ coincides with the direction of the affine normal,
cf.\ \cite{Bla23} and \cite{Lei89}.
Wang \cite{BS00} asked for a classification of locally strongly convex
hypersurfaces having gravity curves which are straight lines.
A weaker condition is the following:
the gravity curve $\gcvatp(t)$ is called flat if its Euclidean
curvature $\kappa(t)$ vanishes at $t=0$.
Note that although $\kappa(t)$ is a Euclidean invariant, the condition
$\kappa(t) = 0$ is affinely invariant.
The affine curvature of $\cc$ is denoted by $\affcu$.

The present author
expressed the flatness of the gravity curve of an affine hypersurface
in terms of its affine invariants \cite{Bin09b}.
Theorem~\ref{tmain1} is a straightforward analog of the main result
of \cite{Bin09b}; Theorem~\ref{tmain2} answers Wang's question for
curves in the plane. More precisely, we show

\begin{theorem}\label{tmain1}
Let $\cc:I\to\R^2$ be a non-degenerate
plane affine curve parametrized by affine arclength. Fix $p\in I$.
Then the gravity curve $\gcvatp$ of $\cc$ at $p$ is flat if and only if
the affine curvature function $\affcu$ satisfies $\affcu'(p)=0$.
\end{theorem}

\begin{theorem}\label{tmain2}
Let $\cc:I\to\R^2$ be a non-degenerate
plane affine curve parametrized by affine arclength. Fix $p\in I$.
Then the gravity curve $\gcvatp$ of $\cc$ at $p$ is a straight line
if and only if the affine curvature function $\affcu$ satisfies
$\affcu(p-s)=\affcu(p+s)$ for all $s\in\R$ sufficiently small.
\end{theorem}

\begin{corollary}
In the setting of Theorem~\ref{tmain2}, $\gcvatp$ is
a straight line for all $p\in I$ if and only if
$\affcu=\const$, i.e.\ $\cc$ is quadratic.
\end{corollary}

It is well known \cite{ST94} that a non-degenerate plane affine curve
with $\affcu=\const$ is a parabola ($\affcu=0$), an ellipsis
($\affcu>0$), or a hyperbola ($\affcu<0$). Differently set,
we get all curves of the form
\[
  a_{11} x^2 + 2a_{12} xy + a_{22} y^2 + 2b_1 x + 2b_2 y = \const,
\]
where the sign of $\affcu$ coincides with the sign of the determinant
$a_{11} a_{22} - a_{12}^2$.

\section{Non-degenerate plane affine curves}

Let $[\cdot,\cdot]$ denote the determinant of $\R^2$, and
let $e_1$, $e_2$ denote its canonical basis.
A differentiable curve $\cc:I\to\R^2$ is called non-degenerate
if $\cc'$ and $\cc''$ are linearly independent. In this case,
the affine arclength $s$ is uniquely determined by $[\cc', \cc''] = 1$. 
Consequently $[\cc', \cc'''] = 0$, and there exists a function
$\affcu:I\to\R$, called affine curvature of $\cc$, such that
\[
  \cc'''(s) = -\affcu(s) \cc'(s),\qquad s\in I.
\]
A detailed introduction to plane affine curves can be found
in \cite{ST94}.
Now fix $\ipt\in I$. After applying a suitable affine transformation
of $\R^2$ we may assume that
\[
  \cc(\ipt)=0,\qquad
  \cc'(\ipt) = e_1,\qquad
  \cc''(\ipt) = e_2.
\]
Continuing to take derivatives of $\cc$ we get
\begin{align*}
\cc'''  &= -\affcu \cc',\qquad
\cc^{(4)}  = -\affcu' \cc' - \affcu \cc'',\qquad
\cc^{(5)} = (-\affcu'' + \affcu^2) \cc' - 2\affcu' \cc'',\\
\cc^{(6)} &= (-\affcu'''+4\affcu\affcu')\cc' +
  (-3\affcu''+\affcu^2) \cc'',\\
\cc^{(7)} &=
(-\affcu^{(4)} + 4 \affcu'{}^2 + 7\affcu\affcu'' - \affcu^3) \cc' +
(-4\affcu''' + 6\affcu\affcu') \cc'',\\
\cc^{(8)} &=
(-\affcu^{(5)} + 15\affcu'\affcu'' + 11\affcu\affcu''' - 9\affcu^2\affcu') \cc' +
(-5\affcu^{(4)} + 10\affcu'{}^2 + 13\affcu\affcu'' - \affcu^3) \cc'',\quad\ldots
\end{align*}
If we write $\cc^{(\kidx)} = \fpre_\kidx \cc' + \gpre_\kidx \cc''$ for
suitable functions $\fpre_\kidx$, $\gpre_\kidx$, then from
\[
  \cc^{(\kidx+1)} =
  \fpre_\kidx' \cc' + \fpre_\kidx \cc'' + \gpre_\kidx' \cc'' +
  \gpre_\kidx \cc''' =
  (\fpre_\kidx' - \affcu \gpre_\kidx) \cc' + (\fpre_\kidx + \gpre_\kidx') \cc''
\]
we get the recursion
\begin{equation}\label{erecu}
  \fpre_{\kidx + 1} = \fpre_\kidx' - \affcu \gpre_\kidx,\qquad
  \gpre_{\kidx + 1} = \gpre_\kidx' + \fpre_\kidx.
\end{equation}
When evaluating $\cc^{(\kidx)}$ at $\ipt$, $\cc'$ and $\cc''$
are replaced by $e_1$, $e_2$, respectively;
this way we obtain the representation
of  $\cc^{(\kidx)}(p)$ in Cartesian coordinates.
Equivalently, we can express the Taylor expansion of
$\cc(s) = (f(s),g(s))^T$ around $\ipt$ by
\begin{align*}
f(s) &= \sum_{k=1}^\infty f_k s^k = s
- \atp{\frac{\affcu}{6}} s^3 - \atp{\frac{\affcu'}{24}} s^4
+ \atp{\frac{-\affcu'' + \affcu^2}{120}} s^5
+ \atp{\frac{-\affcu''' + 4\affcu\affcu'}{720}} s^6 + \dots,\\
g(s) &= \sum_{k=2}^\infty g_k s^k = \frac12 s^2
- \atp{\frac{\affcu}{24}} s^4 - \atp{\frac{\affcu'}{60}} s^5
+ \atp{\frac{-3\affcu''+\affcu^2}{720}} s^6 + \dots,
\end{align*}
where $f_\kidx:=\fpre_\kidx(p)/k!$ and $g_\kidx:=\gpre_\kidx(p)/k!$.
To construct the gravity curve $\gcvatp$ at $p$ we proceed as follows.
Fix $t>0$. Set $g(s)=t^2$ and invert: $s=v(t)$, where
$v(s)=\sum_{k=0}^\infty v_\kidx s^\kidx$ is another power series.
Computing the coefficients of $v$ from those of $g$ we get
$v_0=0=v_2$, $v_1=\sqrt2$, $v_3=-2\sqrt2 g_4$, $v_4=-4g_5$,
$v_5=2\sqrt2(7g_4^2-2g_6)$, and $v_6=64 g_4 g_5 - 8 g_7$.
We obtain
\[
v(t) =
  \sqrt2 t + \atp{\frac{\affcu}{6\sqrt2}} t^3 +
  \atp{\frac{\affcu'}{15}} t^4 +
  \atp{\frac{8\affcu'' + 9\affcu^2}{240\sqrt2}} t^5 +
  \atp{\frac{2\affcu''' + 11\affcu\affcu'}{315}} t^6 + \dots,
\]
and denoting $h(t) := f(v(t))$ we end up in
\begin{equation}\label{efbarh}
h(t) =
  \sqrt2 t - \atp{\frac{\affcu}{2\sqrt2}} t^3 -
  \atp{\frac{\affcu'}{10}} t^4 -
  \atp{\frac{8\affcu'' + 15\affcu^2}{240\sqrt2}} t^5 -
  \atp{\frac{\affcu''' + 9\affcu\affcu'}{210}} t^6 + \dots.
\end{equation}
The parametrized gravity curve can now be written as
\begin{equation}\label{egcv}
\gcvatp(t) = \Bigl(
  \frac12\bigl(h(t) + h(-t)\bigr), t^2
\Bigr)^T.
\end{equation}

\section{Polynomials and power series}

For fixed $\kidx\in\NZ$ denote the set of all polynomials in
$\affcu$, $\affcu'$, \dots, $\affcu^{(\kidx)}$ by $\rp{\kidx}$.
Each summand in such a polynomial will have the form
\begin{equation}\label{polysummand}
  \alpha \cdot
  (\affcu^{(i_1)})^{a_1} \cdots (\affcu^{(i_N)})^{a_N} \cdot
  (\affcu^{(j_1)})^{b_1} \cdots (\affcu^{(j_M)})^{b_M},
\end{equation}
where $\alpha\in\R$, $N,M\in\NZ$, and $i_1,\dots,i_N$ denote
even integers while $j_1,\dots,j_M$ denote odd integers. The odd degree of
the summand \eqref{polysummand} is the number $d=\sum_{l=1}^M b_l\ge 0$.
Let $\pp{\kidx}\subseteq\rp{\kidx}$ denote the polynomials with all summands
of odd degree $d=0, 2, 4,\dots$.
Let $\qp{\kidx}\subseteq\rp{\kidx}$ denote the polynomials with all summands
of odd degree $d=1, 3, 5,\dots$. It will be convenient to agree
that $\pp{-\kidx}:=\R$, $\qp{-\kidx}:=\{0\}$, $\kidx\in\N$.
Clearly, the constants are contained in $\pp{\kidx}$ but not in $\qp{\kidx}$.
Similarly, we also extend the upper index to $\Z$ by defining
$\ppp{\kidx}{\sigma}:=\pp{\kidx}$
whenever $\sigma$ is even, and  $\ppp{\kidx}{\sigma}:=\qp{\kidx}$ when
$\sigma$ is odd. In this way we have defined $\ppp{\kidx}{\sigma}$ for
all $\kidx,\sigma\in\Z$.

\begin{remark*}
\begin{enumerate}[(i)]
\item The advantage of the notation is illustrated by the
identity
\[
  \ppp{k}{\sigma} \cdot \ppp{l}{\tau} \subseteq
  \ppp{\max(k,l)}{\sigma+\tau},
  \qquad k,l,\sigma,\tau\in\Z,
\]
which is a condensed form of
\[
  \pp{k}\cdot\pp{l}\subseteq\pp{\max(k,l)},\quad
  \pp{k}\cdot\qp{l}\subseteq\qp{\max(k,l)},\quad
  \qp{k}\cdot\qp{l}\subseteq\pp{\max(k,l)}.
\]
\item If $r\in\ppp{\kidx}{\sigma}$, then
$r'\in\ppp{\kidx+1}{\sigma+1}$ and hence
$r^{(l)}\in\ppp{\kidx+l}{\sigma+l}$ for all $l\in\N$.
This is the case since differentiating the factor
$(\affcu^{(i_l)})^{a_l}$ of \eqref{polysummand} creates a new summand
with odd degree one more; differentiating the factor
$(\affcu^{(j_l)})^{b_l}$ creates a new summand with odd degree one less.
\end{enumerate}
\end{remark*}

We consider power series $a(s)=\sum_{\kidx=0}^\infty a_\kidx s^k$
whose coefficients are $a_\kidx$ such polynomials
in the derivatives of $\affcu$.
Such a power series $a(s)$ is called
$(n,\sigma)$-alternating if $a_\kidx\in\ppp{\kidx-n}{\kidx+\sigma}$ for all
$\kidx\in\NZ$. The power series $a(s)$ is called
$n$-explicit, if it is $(n,n)$-alternating and
\[
  a_\kidx = \lk{a}_\kidx \affcu^{(\kidx-n)} + p_\kidx \in
  \ppp{\kidx-n}{\kidx-n},\qquad
  \kidx\in\NZ,
\]
where
$\lk{a}_\kidx\in\R\setminus\{0\}$,
$p_\kidx\in\ppp{\kidx-n-1}{\kidx-n}$ are
arbitrary, and negative derivatives of $\affcu$ are assumed as $0$.
If $a$ is explicit then we call $(\lk{a}_\kidx)_{\kidx=0}^\infty$
its leading coefficients.

\begin{lemma}\label{lsqrt}
Assume $a(s)=\sum_{\kidx=2}^\infty a_\kidx s^\kidx$, $a_2>0$ constant
(not depending on $\affcu$) is given.
Let $b(s)=\sum_{\kidx=1}^\infty b_\kidx s^\kidx$ be such that $(b(s))^2 = a(s)$.
Then the power series $b$ is well defined up to sign. Moreover:
\begin{enumerate}[(i)]
\item If $a$ is $(n,0)$-alternating, then $b$ is $(n-1,1)$-alternating.
\item If $a$ is $n$-explicit, then $b$ is $(n-1)$-explicit.
In this case, the corresponding leading coefficients satisfy
$\lk{b}_\kidx = \pm\lk{a}_{\kidx+1} / (2 \sqrt{a_2})$.
\end{enumerate}
\end{lemma}

\begin{proof}
The coefficients of $b$ can be computed from those of $a$
via $b_1^2 = a_2$ and
\begin{equation}\label{esqrt}
  b_\kidx = \frac{1}{2b_1}
  \Bigl(
    a_{\kidx+1} - \sum_{l=2}^{\kidx-1} b_l b_{\kidx+1-l}
  \Bigr),
  \qquad \kidx=2,3,\dots
\end{equation}
The choice of $b_1=\pm\sqrt{a_2}$ gives rise to two distinct solutions
for $b$, differing only by sign. We show (i) by induction on $\kidx$,
where the assertion is true for $\kidx=1$.
For fixed $\kidx>1$ we have $a_{\kidx+1}
\in\ppp{\kidx-(n-1)}{\kidx+1}$
and by assumption
\[
  b_l\, b_{\kidx+1-l} \in \ppp{l-(n-1)}{l+1} \cdot
  \ppp{\kidx+1-l-(n-1)}{\kidx+2-l} \subseteq
  \ppp{\kidx-1-(n-1)}{\kidx+3} =
  \ppp{\kidx-n}{\kidx+1},
\]
and therefore $b_\kidx\in\ppp{\kidx-(n-1)}{\kidx+1}$ from \eqref{esqrt}.
Assertion (ii) is also clear from these formulas;
the highest derivative of $\affcu$ in $b_\kidx$ comes from $a_{\kidx+1}$.
As a side note, since $a_2$ is a nonzero constant,
$a$ cannot be $(n,1)$-alternating for any $n\in\Z$.
\end{proof}

In the following we summarize basic facts regarding
the composition of power series, cf.~\cite{Com74}. Suppose
$a(s)=\sum_{\kidx=1}^\infty a_\kidx s^\kidx$, $b(t)=\sum_{\kidx=1}^\infty b_\kidx t^\kidx$
are two power series. Then their composition $\comporesult(s) := b(a(s))$
is again a power series since its coefficients can be expressed as
\begin{equation}\label{ecompo}
  \comporesult_\kidx =
  \sum_{l=1}^\kidx b_l B_{\kidx,l}(a_1,\dots,a_{\kidx-l+1}),\qquad \kidx\in\N,
\end{equation}
where $B_{\kidx,l}$ is a Bell polynomial of $a=(a_1,a_2,\dots)$ given by
\[
  B_{k,l}(a_1,\dots,a_{k-l+1}) = \sum_{(j_1,\dots,j_{k-l+1})}
  \frac{k!}{j_1! \cdots j_{k-l+1}!}
  \Bigl(\frac{a_1}{1!}\Bigr)^{j_1}
  \dots
  \Bigl(\frac{a_{k-l+1}}{(k-l+1)!}\Bigr)^{j_{k-l+1}},
\]
where the sum runs over all tuples of non-negative integers satisfying
$j_1+j_2+\dots+j_{k-l+1}=l$ and $j_1+2j_2+\dots+(k-l+1)j_{k-l+1}=k$
(assume $0^0=1$ for the sake of this formula). Bell polynomials can
also be given by
\begin{equation}\label{econvid}
  l!\cdot B_{k,l}(a_1,\dots,a_{k-l+1}) = \overbrace{(a \ast a \ast \dots
      \ast a)}^{\text{$l$ copies}}{}_k =: a^{\ast l}_k, 
\end{equation}
where $\ast$ is the weighted convolution defined by
\begin{equation}\label{econv}
  (a\ast b)_k = \sum_{l=1}^{k-1} \binom{k}{l} a_l b_{k-l},
\end{equation}
whenever $a=(a_1,a_2,\dots)$ and $b=(b_1,b_2,\dots)$.

\begin{lemma}\label{lcompo}
Suppose
$a(s)=\sum_{\kidx=1}^\infty a_\kidx s^\kidx$,
$b(t)=\sum_{\kidx=1}^\infty b_\kidx t^\kidx$
are two power series.
Let $\comporesult(s) := b(a(s))$ be their composition.
\begin{enumerate}[(i)]
\item If $a$ is $(n,1)$-alternating and
$b$ is $(m,\sigma)$-alternating, then
$\comporesult$ is $(\min(n+1,m),\sigma)$-alternating.
\item Fix $n\in\N$ odd, $n\ge 3$.
Suppose $a$ and $b$ are both $n$-explicit and
$\lk{\comporesult}_\kidx = b_1 \lk{a}_\kidx + a_1^\kidx \lk{b}_\kidx \not= 0$
for all $\kidx\in\N$,
then $\comporesult$ is also $n$-explicit, with leading coefficients
$\lk{\comporesult}_\kidx$.
\end{enumerate}
\end{lemma}

\begin{proof}
If $a$ is $(n,\sigma)$-alternating and $b$ is $(m,\tau)$-alternating,
then $a\ast b$ is $(\min(n,m)+1,\sigma+\tau)$-alternating.
This is immediately clear from
\[
  a_l \, b_{\kidx-l} \in \ppp{l-n}{l+\sigma} \cdot
  \ppp{\kidx-l-m}{\kidx-l+\tau}
  \subseteq \ppp{\kidx-1-\min(n,m)}{\kidx+\sigma+\tau}
\]
and the definition \eqref{econv}.
Hence $a^{\ast l}$ is $(l(n+1),l)$-alternating.
Now assume that $\comporesult$ is given by \eqref{ecompo}.
\[
  b_l \, a^{\ast l}_\kidx \in
  \ppp{l-m}{l+\sigma} \cdot \ppp{l-l(n+1)}{l+\kidx} \subseteq
  \ppp{\kidx-\min(n+1,m)}{\kidx+\sigma},
\]
which proves (i) by the convolution identity \eqref{econvid}.
To show (ii) first observe that since $n\ge 3$, $a_1$ and $b_1$ do not
depend on $\affcu$, i.e.\ they are constants.
Therefore $\lk{\comporesult}_\kidx$ in (ii) is
well defined. From (i) $\comporesult$ is $(n,n)$-alternating.
The highest derivative $\affcu^{(\kidx-n)}$ in
$\comporesult_\kidx$ comes about when $l=1$ or $l=\kidx$ in \eqref{ecompo}.
Since $B_{\kidx,1}(a_1,\dots,a_\kidx)=a_\kidx$ and $B_{\kidx,\kidx}(a_1)=a_1^\kidx$,
\[
  \comporesult_\kidx
  \,=\, b_1 a_\kidx + b_\kidx a_1^\kidx + \ldots
  \,=\, \lk{\comporesult}_\kidx \affcu^{(\kidx-n)} + \ldots,
\]
where the dots indicate terms containing only derivatives of $\affcu$
lower than $\kidx-n$.
\end{proof}

\begin{lemma}\label{lcompinv}
Suppose $t=a(s)=\sum_{\kidx=1}^\infty a_\kidx s^\kidx$, $a_1=\const\not=0$ is given.
Then $s=b(t)$, the composition inverse of $a$, is well defined.
Moreover:
\begin{enumerate}[(i)]
\item If $a$ is $(n,1)$-alternating, then $b$ is also $(n,1)$-alternating.
\item If $a$ is $n$-explicit for $n$ odd, then $b$ is also $n$-explicit.
In this case, the corresponding leading coefficients satisfy
$\lk{b}_\kidx = - a_1^{-\kidx-1} \lk{a}_\kidx$.
\end{enumerate}
\end{lemma}

\begin{proof}
In \eqref{ecompo} assume that $a$ is given,
$\comporesult_1=1$ and $\comporesult_\kidx=0$
for all $\kidx\not=1$, and solve for $b$, the composition inverse of $a$.
We get $b_0=0$, $b_1 = 1/a_1$, and
\[
  b_\kidx = - \frac{1}{a_1^\kidx} \Bigl(
    b_1 a_\kidx +
    \sum_{l=2}^{\kidx-1} b_l
    B_{\kidx,l}(a_1,a_2,\dots,a_{\kidx-l+1})
  \Bigr),\qquad \kidx\ge 2,
\]
since $B_{\kidx,1}(a_1,\dots,a_\kidx)=a_\kidx$ and
$B_{\kidx,\kidx}(a_1)=a_1^\kidx$. Assertions (i) and (ii) are immediately
clear from this formula when applying similar arguments as in
the proof of Lemma~\ref{lcompo}.
\end{proof}

\section{Proof of the theorems}

Before we complete the proofs we need to establish the explicitness
of the components of $\cc$.

\begin{lemma}\label{lfg}
The Taylor expansions of $f(s)$ and $g(s)$ are
$3$-explicit and $4$-explicit, respectively. More precisely,
for arbitrary $\kidx\in\NZ$,
\begin{align*}
\kidx! \, f_\kidx &= -\affcu^{(\kidx-3)} + p_\kidx
\in \ppp{\kidx-3}{\kidx+1},\\
\kidx! \, g_\kidx &= -(\kidx-3) \affcu^{(\kidx-4)} + q_\kidx 
\in \ppp{\kidx-4}{\kidx},
\end{align*}
for some $p_\kidx\in \ppp{\kidx-5}{\kidx+1}$ and
$q_\kidx\in \ppp{\kidx-6}{\kidx}$.
In particular, the leading coefficients satisfy
$\lk{f}_\kidx=-1/\kidx!$ and
$\lk{g}_\kidx=-(\kidx-3)/\kidx!$.
\end{lemma}

\begin{proof} Use induction on $\kidx$ in \eqref{erecu}.
For $\kidx=0,\dots,4$ the assertion is true.
Assume it is true for $\kidx-1$, i.e.\ there exist
$p_{\kidx-1}\in\ppp{\kidx-6}{\kidx}$ and
$q_{\kidx-1}\in\ppp{\kidx-7}{\kidx-1}$ such that
\begin{align*}
  (\kidx-1)!\, f_{\kidx-1}  &=
    -\affcu^{(\kidx-4)} + p_{\kidx-1},\qquad\qquad
  (\kidx-1)!\, f_{\kidx-1}' =
    -\affcu^{(\kidx-3)} + p_{\kidx-1}',\\
  (\kidx-1)!\, g_{\kidx-1}  &=
    -(\kidx-4) \affcu^{(\kidx-5)} + q_{\kidx-1},\quad
  (\kidx-1)!\, g_{\kidx-1}' =
    -(\kidx-4) \affcu^{(\kidx-4)} + q_{\kidx-1}'.
\end{align*}
Then from \eqref{erecu} we get the desired assertion where
it is easy to check that
\begin{align*}
  p_\kidx &:= p_{\kidx-1}' - (\kidx-4)\affcu\affcu^{(\kidx-5)} -
    \affcu q_{\kidx-1} \in \ppp{\kidx-5}{\kidx+1},\\
  q_\kidx &:= q_{\kidx-1}' + p_{\kidx-1} \in \ppp{\kidx-6}{\kidx},
\end{align*}
which finishes the proof.
\end{proof}

\begin{proof}[Proof of Theorem~\ref{tmain1}]
From \cite[Sect.~5]{Bin09b} it is known that $\gcvatp$
is flat if and only if the coefficient of $t^4$ in $h(t)$ vanishes.
From \eqref{efbarh}, \eqref{egcv} this condition translates to $\affcu'(p)=0$.
The computations of \cite{Bin09b} carry over with hardly any
change to curves in $\R^2$; we renounce to include a line-by-line copy.
\end{proof}

\begin{proof}[Proof of Theorem~\ref{tmain2}]
Since by Lemma~\ref{lfg} $g(s)$ is $4$-explicit,  
by Lemma~\ref{lsqrt}, $u(s)$ in $(u(s))^2=g(s)$ exists and
is $3$-explicit. By Lemma~\ref{lcompinv} the composition
inverse $v(t)=s$ of $u(s)=t$ exists and
is also $3$-explicit. Finally, the composition $h(t)=f(v(t))=f(s)$
exists and is $3$-explicit as a power series in $t$ by
Lemma~\ref{lcompo}, where it remains to check that
\[
  \lk{h}_\kidx = f_1 (-u_1^{-\kidx-1} \lk{u}_\kidx) + v_1^\kidx \lk{f}_\kidx
  = - \frac{3\sqrt{2}^\kidx}{(\kidx+1)!} \not= 0,\quad \kidx\in\N,
\]
where we used Lemma~\ref{lfg}, $f_1=1$, $v_1=\sqrt2$, $u_1=1/\sqrt2$,
and $\lk{u}_\kidx = \lk{g}_{\kidx+1}/\sqrt2$ since $g_2=1/2$.

Now assume that $\gcvatp$ is a straight line, or
$\gcvatp(t)=(0,t^2)^T$ from \eqref{egcv}.
Then $h(t)$ is an odd function, or $h_{2j}=0$ for all $j\in\NZ$,
where the $h_\kidx$ are viewed as reals,
i.e.\ polynomials in $\affcu$ evaluated at $p$.
Since $h$ is 3-explicit, we get
$\affcu^{(2j+1)}(p) = 0$ by a simple induction on $j$.
Therefore, $\affcu$ is even about $p$.
All arguments used work also for the converse direction.
\end{proof}

\newcommand{\litdir}{/home/tom/geom/lit}
\renewcommand{\litdir}{/cygdrive/c/tbinder/Dropbox/work/adglit}
\bibliographystyle{amsalpha}
\bibliography{\litdir/all/Macros,\litdir/all/Articles,%
\litdir/all/In-Thesis,\litdir/all/Books,\litdir/all/Bibliographies}
\end{document}